\documentclass[a4paper]{article}

\usepackage[T1]{fontenc}
\usepackage{amsmath,amssymb,amsthm}
\usepackage{cite,mathtools}

\usepackage[pdftex,bookmarksopen=true,bookmarks=true,unicode,setpagesize]{hyperref}
\hypersetup{colorlinks=true,linkcolor=black,citecolor=black,urlcolor=blue}

\allowdisplaybreaks[3]
\numberwithin{equation}{section}

\makeatletter
\renewcommand*{\@fnsymbol}[1]{\ensuremath{\ifcase#1\or 1\or 2\or
   3\else\@ctrerr\fi}}
\makeatother

\newtheorem{theorem}{Theorem}
\newtheorem{proposition}{Proposition}[section]
\theoremstyle{remark}
\newtheorem{remark}[proposition]{Remark}

\renewcommand{\Re}{\mathrm{Re}\,}
\renewcommand{\Im}{\mathrm{Im}\,}

\newcommand{\R}{\mathbb{R}}
\newcommand{\eps}{\varepsilon}
\newcommand{\la}{\lambda}
\newcommand{\An}{\mathcal{H}}  

\title{An Ikehara-type theorem for~functions convergent to zero}

\author{Dmitri Finkelshtein\thanks{Department of Mathematics, Swansea University, Singleton Park, Swansea SA2 8PP, U.K. ({\tt d.l.finkelshtein@swansea.ac.uk}).} \and Pasha Tkachov\thanks{Gran Sasso Science Institute, Viale Francesco Crispi, 7, 67100 L'Aquila AQ, Italy ({\tt pasha.tkachov@gssi.it}).}}

\begin{document}
\maketitle

\begin{abstract}
We prove an analogue of the Ikehara theorem for positive non-increasing functions convergent to zero, generalising the results postulated in Diekmann,\,Kaper (1978) \cite{DK1978} and Carr,\,Chmaj (2004) \cite{CC2004}.



\end{abstract}

\section{Introduction} 

The Ikehara theorem and its extensions are the so-called complex Tauberian theorems, inspired, in particular, by the number theory, see e.g. the review \cite{Kor2002}. The following version of the Ikehara theorem can be found in \cite[Subsection~2.5.7]{EE1985}: 
\begin{theorem}\label{thm:I}
  Let $\phi$ be a positive monotone increasing function, and let there exist $\mu>0$, $j>0$, such that 
  \begin{equation}\label{eq:laplace_of_a}
    \int_{0}^{\infty} e^{-tz}\phi(t)dt = \frac{F(z)}{(z-\mu)^{j}}, \quad \Re z>\mu,
  \end{equation}
  where $F$ is holomorphic on $\{\Re z\geq \mu\}$.
  Then, for some $D>0$,
  \[
    \phi(t)\sim \frac{D}{\Gamma(j)} t^{j-1} e^{\mu t}, \quad t\to\infty.
  \]
\end{theorem}
Alternatively, Theorem~\ref{thm:I} may be formulated for the Stieltjes measure $d\phi(t)$ instead of $\phi(t)dt$ obtaining similar asymptotic for $\phi$ (see e.g. Proposition~\ref{lem:EII} below). 

In Theorem~\ref{thm:I}, $\phi$ increases to $\infty$. In~\cite[Lemma~6.1]{DK1978} (for $j=1$) and in~\cite[Proposition~2.3]{CC2004} (for $j>0$), the similar results were stated for positive monotone decreasing $\varphi$ (cf., correspondingly, Propositions~\ref{prop:Diekmann_Kaper} and \ref{prop:Carr_Chmaj} below). The aim of both generalizations was to find an \emph{a~priori} asymptotics for solutions to a class of nonlinear integral equations. In~\cite{CC2004}, in particular, it was applied to the study of the uniqueness of traveling wave solutions to certain nonlocal reaction-diffusion equations; see also e.g.~\cite{WLX2018,TPT2016,LZ2016,CDM2008,ZLW2012}. Note also that then, the case $j=2$ corresponded to the traveling wave with the minimal speed.

In both papers \cite{DK1978,CC2004} no proof was given, mentioning that it is supposed to be analogous to the case of increasing $\phi$ without any further details. In~Theorem~\ref{thm:tauber} below, we prove an analogue of Theorem~\ref{thm:I} for non-increasing function, and in Proposition~\ref{prop:Carr_Chmaj} we apply it to prove the mentioned result of~\cite{CC2004}. We require, however, an {\em a~priori} regular decaying of $\varphi$, namely, we assume that there exists $\nu>0$, such that $\varphi(t)e^{\nu t}$ is an increasing function.
We require also the convergence of  $\int_0^\infty e^{zt}d\varphi(t)$ for $0<\Re z<\mu$ instead of the weaker corresponding assumption for $\int_0^\infty e^{zt}\varphi(t)dt$.

Beside the aim to present a proof, the reason for the generalization we provide was to omit the requirement on the function $F$ to be analytical on the line $\{\Re z=\mu\}$ keeping the general case $j>0$. We were motivated by the integro-differential equation we studied in \cite{FKT100-2} (which covers the equations considered in \cite{CC2004}), where the Laplace-type transform of the traveling wave with the minimal speed (that requires, recall, $j=2$)
 might be not analytical at $z=\mu$.

Our result is based on a version of the Ikehara--Ingham theorem proposed in \cite{Ten1995}, see Proposition~\ref{lem:EII} below. Using the latter result, we prove also in Proposition~\ref{prop:Diekmann_Kaper} a generalization of \cite[Lemma~6.1]{DK1978}  (under the regularity assumptions on $\varphi$ mentioned above).

\section{Main results}
Let, for any $D\subset\mathbb{C}$, $\An(D)$
 be the class of all holomorphic functions on $D$.
\begin{theorem}\label{thm:tauber}
Let $\varphi:\R_+\to  \R_+:=[0,\infty)$ be a non-increasing function such that, for some $\mu>0$, $\nu>0$,
\begin{equation}\label{assum:1}
    \text{the function } e^{\nu t} \varphi(t) \text{ is non-decreasing,}
\end{equation}
and 
\begin{equation}\label{eq:psi_onesided_lap}
    \int\limits_0^{\infty}e^{z t}d\varphi(t)<\infty, \quad 0< \Re z< \mu. 
\end{equation}
Let also the following assumptions hold.
\begin{enumerate}
  \item There exist a constant $j>0$ and complex-valued functions 
    \[
      H\in\An(0<\Re z\leq \mu), \qquad  F\in\An (0<\Re z<\mu)\cap C(0<\Re z\leq \mu),
    \]
    such that the following representation holds
    \begin{equation}\label{eq:sing_repres}
      \int\limits_0^{\infty}e^{z t}\varphi(t)dt=\dfrac{F(z)}{(\mu-z)^j}+H(z),\quad 0<\Re z<\mu.
    \end{equation}
  \item For any $T>0$,
    \begin{equation}\label{loglim}
      \lim_{\sigma\to0+}g_j(\sigma)\sup_{|\tau|\leq T}\bigl\lvert F(\mu-2\sigma-i\tau)-F(\mu-\sigma-i\tau)\bigr\rvert=0,
    \end{equation}
  where, for $\sigma>0$,
  \begin{equation}\label{eq:def_gj}
      g_j(\sigma):=\begin{cases}
    \sigma^{j-1}, & 0<j<1,\\
    \log \sigma, & j=1,\\
    1, &j>1.
    \end{cases}
  \end{equation}
\end{enumerate}
Then $\varphi$ has the following asymptotic
\begin{equation}\label{eq:psi_asympt}
  \varphi(t)\sim \frac{F(\mu)}{\Gamma(j)} t^{j-1}e^{-\mu t} ,\quad t\to \infty.
\end{equation}
\end{theorem}

The proof of Theorem~\ref{thm:tauber} is based on the following Tenenbaum's result.
\begin{proposition}[{``Effective'' Ikehara--Ingham Theorem, cf.~{\cite[Theorem 7.5.11]{Ten1995}}}]\label{lem:EII}
Let $\alpha(t)$ be a non-decreasing function such that, for some fixed $a>0$, the following integral converges:
\begin{equation}\label{eq:onesided_laplace}
\int\limits_0^{\infty}e^{-z t}d\alpha(t), \quad \Re z>a.
\end{equation}
Let also there exist constants $D\geq0$ and $j>0$, such that for the functions
\begin{align}\label{eq:G}
G(z)&:=\dfrac{1}{a+z}\int\limits_0^\infty e^{-(a+z)t}d\alpha(t)-\dfrac{D}{z^{j}},\quad \Re z>0,\\
\eta(\sigma,T)&:=\sigma^{j-1}\int\limits_{-T}^{T}\bigl\lvert G(2\sigma+i\tau)-G(\sigma+i\tau)\bigr\rvert d\tau, \quad T>0,\label{eq:etaform}
\end{align}
one has that
\begin{equation}\label{eq:eta}
\lim_{\sigma\to0+}\eta(\sigma,T)=0, \quad T>0.
\end{equation}
Then
\begin{equation}\label{eq:alpha}
\alpha(t)=\left\{ \dfrac{D}{\Gamma(j)}+O\bigl(\rho(t)\bigr)\right\}  e^{at}t^{j-1},\quad t\geq1,
\end{equation}
where
\begin{equation}\label{eq:rho}
\rho(t):=\inf\limits_{T\geq32(a+1)} \Bigl\{ T^{-1}+\eta\bigl(t^{-1},T\bigr)+(Tt)^{-j} \Bigr\}.
\end{equation}
\end{proposition}
\begin{proof}[Proof of Theorem~\ref{thm:tauber}]
We first express $\int_{0}^{\infty}e^{\la t}\varphi(t)dt$ in the form \eqref{eq:onesided_laplace}. Fix any $a>0$ such that $\mu+a>\nu$. Then, 
by \eqref{assum:1}, the function 
\begin{equation}\label{eq:alpha-new}
      \alpha(t):= e^{(\mu+a)t}\varphi(t), \quad t>0,
\end{equation}
is increasing. Since $\varphi$ is monotone, then, for
 any $0<\Re z<\mu$, one has
\begin{equation}\label{eq:psi_iprime}
\int\limits_0^\infty e^{-(a+z)t}d\alpha(t)=(\mu+a)\int\limits_0^{\infty}e^{(\mu-z)t}\varphi(t)dt +\int\limits_0^\infty e^{(\mu-z)t}d\varphi(t),
\end{equation}
where both integrals in the right hand side of \eqref{eq:psi_iprime} converge, for $0<\Re z<\mu$, because of \eqref{eq:psi_onesided_lap}--\eqref{eq:sing_repres}.
 
Then, by \cite[Corollary~II.1.1a]{Wid1941}, the integral in the 
left hand side of \eqref{eq:psi_iprime} converges, for \emph{all} $\Re z>0$. 
Therefore, by \cite[Theorem II.2.3a]{Wid1941}, one gets 
another representation for the latter integral, for $\Re z>0$:
\begin{equation}\label{eq:psi_i}
\int\limits_0^\infty e^{-(a+z)t}d\alpha(t)=-\varphi(0)+(a+z)\int\limits_0^\infty e^{(\mu-z)t}\varphi(t)dt.
\end{equation}

Let $G$ be given by \eqref{eq:G} with $\alpha(t)$ as above and $D:=F(\mu)$. 
Combining \eqref{eq:psi_i} with \eqref{eq:sing_repres} (where we replace $z$ by $\mu-z$),
we obtain, for $0<\Re z<\mu$,
\begin{align}
G(z)&=\dfrac{F(\mu-z)}{z^j}+ K(z), \\ K(z)&:= H(\mu-z)-\dfrac{\varphi(0)}{a+z}-\frac{F(\mu)}{z^j}.
\end{align}

Check the condition \eqref{eq:eta}; one can assume, clearly, that 
$0<\sigma<\frac{\mu}{2}$. 
 Since 
$K\in\An(0<\Re z <\mu)$,
 one easily gets that
\begin{align}
&\quad \lim\limits_{\sigma\to  0+} \sigma^{j-1}\int\limits_{-T}^{T} \bigl\lvert G(2\sigma+i\tau)-G(\sigma+i\tau)\bigr\rvert d\tau \notag \\
 &\leq \lim\limits_{\sigma\to  0+}\sigma^{j-1}\int\limits_{-T}^{T}\Bigl\lvert \dfrac{F(\mu-2\sigma-i\tau)-F(\mu)}{(2\sigma+i\tau)^j}-\dfrac{F(\mu-\sigma-i\tau)-F(\mu)}{(\sigma+i\tau)^j}\Bigr\rvert d\tau\notag\\
 &\leq \lim\limits_{\sigma\to  0+}\sigma^{j-1}\int\limits_{-T}^{T}\Bigl\lvert \dfrac{F(\mu-2\sigma-i\tau)-F(\mu-\sigma-i\tau)}{(\sigma+i\tau)^j}\Bigr\rvert d\tau\notag\\
 &\quad+\lim\limits_{\sigma\to  0+}\sigma^{j-1}\int\limits_{-T}^{T}\bigl\lvert F(\mu-2\sigma-i\tau)-F(\mu)\bigr\rvert \Bigl\lvert \dfrac{1}{(2\sigma+i\tau)^j}-\dfrac{1}{(\sigma+i\tau)^j}\Bigr\rvert d\tau\notag,\\
 &=: \lim\limits_{\sigma\to  0+} A_j(\sigma)+\lim\limits_{\sigma\to  0+} B_j(\sigma).\label{ABexpans}
 \end{align}

Prove that both limits in \eqref{ABexpans} are equal to $0$. For each $j>0$, we define the function 
\begin{equation}\label{eq:ssa}
    h_j(\sigma):=\sigma^{j-1}
    \int_{-T}^T\frac{1}{(\sigma^2+\tau^2)^\frac{j}{2}}d\tau, \qquad \sigma>0.
\end{equation}
We have then
 \begin{equation}\label{eq:ineqforAs}
    A_j(\sigma) \leq \sup_{|\tau|\leq T}\bigl\lvert F(\mu-2\sigma-i\tau)-F(\mu-\sigma-i\tau)\bigr\rvert \ h_j(\sigma).
 \end{equation}

It is straighforward to check that
\begin{equation}\label{eq:as-h1}
    h_1(\sigma)= 2\log \dfrac{\sqrt{T^2+\sigma^2}+T}{\sigma} \sim -2\log \sigma, \quad \sigma\to0+.
\end{equation}
For $j\neq 1$, we make the substitution $\tau=\sigma \tan t$ in \eqref{eq:ssa}, then
\begin{equation}\label{eq:sadsa22}
    h_j(\sigma)=
 \int_{-\arctan\frac{T}{\sigma}}^{\arctan\frac{T}{\sigma}}(\cos t)^{j-2}\, d t.
\end{equation}
Therefore, 
\begin{equation}\label{eq:as-h2}
    h_j(\sigma)\leq 2\arctan\frac{T}{\sigma}<\pi, \qquad \sigma>0, \ j\geq 2.
\end{equation}
Let now $0<j<2$, $j\neq 1$. Then replacing $t$ by $\frac{\pi}{2}-t$ in \eqref{eq:sadsa22}, we obtain
\begin{equation}\label{eq:hsadhdsa}
    h_j(\sigma)=2\int_{\frac{\pi}{2}-\arctan\frac{T}{\sigma}}^{\frac{\pi}{2}}(\sin t)^{j-2}\,dt\leq 2^{3-j}\int_{\frac{\pi}{2}-\arctan\frac{T}{\sigma}}^{\frac{\pi}{2}} t^{j-2}\,dt,
\end{equation}
since $\sin t > \frac{t}{2}$, $t\in(0,\frac{\pi}{2}]$. Therefore, 
\begin{equation}\label{eq:as-h12}
    h_j(\sigma)\leq 2^{3-j}\int_{0}^{\frac{\pi}{2}} t^{j-2}\,dt=\frac{2^{4-2j}\pi^{j-1}}{j-1}, \quad \sigma>0, \ 1<j<2.
\end{equation}
Finally, for $0<j<1$, we obtain from \eqref{eq:hsadhdsa}, that
\begin{equation}\label{eq:sdsadfewr}
    h_j(\sigma)\leq \frac{2^{3-j}}{1-j} \biggl( 
  \Bigl( \frac{\pi}{2}-\arctan\frac{T}{\sigma}\Bigr)^{j-1}-\Bigl( \frac{\pi}{2}\Bigr)^{j-1}\biggr)
\end{equation}
Since $\arctan x+\arctan\frac{1}{x}=\frac{\pi}{2}$, $x>0$, we get
\[
  \arctan\frac{T}{\sigma}=\frac{\pi}{2}-\arctan\frac{\sigma}{T}\sim\frac{\pi}{2}-\frac{\sigma}{T}, \quad \sigma\to0+.
\]
Therefore, \eqref{eq:sdsadfewr} implies that 
\begin{equation}\label{eq:as-h01}
    h_j(\sigma)=O(\sigma^{j-1}), \quad \sigma\to0+, \ 0<j<1.
 \end{equation}

Combining \eqref{eq:as-h1}, \eqref{eq:as-h2}, \eqref{eq:as-h12}, \eqref{eq:as-h01} with \eqref{eq:def_gj}, we have that 
\begin{equation}\label{eq:cool}
      h_j(\sigma)=O(g_j(\sigma)), \quad \sigma\to0+, \ j>0,
\end{equation}
that, together with \eqref{eq:ineqforAs} and \eqref{loglim}, yield $\lim\limits_{\sigma\to  0+} A_j(\sigma) =0$.

Take now an arbitrary $\beta\in(0,\mu)$ and consider, for each $T>0$, the set 
\begin{equation}\label{eq:Kset}
    K_{\beta,\mu,T}:=\bigl\{z\in\mathbb{C} \bigm| \beta\leq \Re z \leq \mu,\ \lvert \Im z \rvert \leq T\bigr\}.
\end{equation}
Let $0<\sigma<\mu^2$; since 
 $F\in C(K_{\sqrt{\sigma},\mu,T})$, there exists $C_1>0$ such that $|F(z)|\leq C_1$, $z\in K_{\sqrt{\sigma},\mu,T}$. Therefore,  
 \begin{align}
   B_j(\sigma) & \leq \sigma^{j-1}\sup_{|\tau|\leq\sqrt{\sigma}}\bigl\lvert F(\mu-2\sigma-i\tau)-F(\mu)\bigr\rvert  \int\limits_{|\tau|\leq\sqrt{\sigma}} \Bigl\lvert\dfrac{1}{(2\sigma+i\tau)^j}-\dfrac{1}{(\sigma+i\tau)^j}\Bigr\rvert d\tau \notag\\
    & \quad+ 2C_1 \sigma^{j-1}\int\limits_{\sqrt{\sigma}\leq|\tau|\leq T}\Bigl\lvert \dfrac{1}{(2\sigma+i\tau)^j}-\dfrac{1}{(\sigma+i\tau)^j}\Bigr\rvert d\tau.\label{Bsigmaest}
\end{align}
 
Next, since
\begin{align*}
j \biggl\lvert \dfrac{1}{(2\sigma+i\tau)^j}-\dfrac{1}{(\sigma+i\tau)^j}\biggr\rvert
&=\biggl\lvert \int_{\sigma+i\tau}^{2\sigma+i\tau}\dfrac{1}{z^{j+1}}\,dz \biggr\rvert
=\biggl\lvert \int_0^1\dfrac{\sigma}{\bigl((1+t)\sigma+i \tau\bigr)^{j+1}}\,dt \biggr\rvert
\\&\leq \int_0^1\frac{\sigma}{\bigl((1+t)^2\sigma^2+\tau^2\bigr)^{\frac{j+1}{2}}}\,dt\leq
\frac{\sigma}{\bigl(\sigma^2+\tau^2\bigr)^{\frac{j+1}{2}}},
\end{align*}
we can continue \eqref{Bsigmaest} as follows, cf. \eqref{eq:ssa},
\begin{align}
  j B_j(\sigma) & \leq \sup_{|\tau|\leq\sqrt{\sigma}}\bigl\lvert F(\mu-2\sigma-i\tau)-F(\mu)\bigr\rvert \  h_{j+1}(\sigma) \notag\\
    & \quad+ 4C_1 \int\limits_{\sqrt{\sigma}\leq \tau \leq T}\frac{\sigma^j}{\bigl(\sigma^2+\tau^2\bigr)^{\frac{j+1}{2}}} d\tau.\label{Bsigmaest2}
\end{align}

By \eqref{eq:as-h2} and \eqref{eq:as-h12}, functions $h_{j+1}$ are bounded on $(0,\infty)$ for all $j>0$. Next, since $F$ is uniformly continuous on $K_{\sqrt{\sigma},\mu,T}$, we have that, for any $\eps>0$ there exists $\delta>0$ such that
$f(\mu,\sigma,\tau):=\bigl\lvert F(\mu-2\sigma-i\tau)-F(\mu)\bigr\rvert<\eps$,
if only $4\sigma^2+\tau^2<\delta$. Therefore, if $\sigma>0$ is such that $4\sigma^2+\sigma<\delta$ then $\sup_{|\tau|\leq\sqrt{\sigma}}f(\mu,\sigma,\tau)<\eps$ hence
\begin{equation}\label{eq:sdasd111}
    \sup_{|\tau|\leq\sqrt{\sigma}}\bigl\lvert F(\mu-2\sigma-i\tau)-F(\mu)\bigr\rvert \ h_{j+1}(\sigma)\to 0, \quad \sigma\to0+.
\end{equation}
Finally, making again the substitution $\tau=\sigma \tan t$ in the integral in \eqref{Bsigmaest2}, we obtain that it is equal to
\[
  I_j:=\int_{\arctan\frac{\sqrt{\sigma}}{\sigma}}^{\arctan\frac{T}{\sigma}}(\cos t)^{j-1}\, d t.
\]
Similarly, to the above, for $j\geq 1$, 
\[
  I_j\leq \arctan\frac{T}{\sigma}-\arctan\frac{\sqrt{\sigma}}{\sigma},
\]
and, for $0<j<1$,
\[
  I_j=\int_{\frac{\pi}{2}-\arctan\frac{T}{\sigma}}^{\frac{\pi}{2}-\arctan\frac{\sqrt{\sigma}}{\sigma}} \frac{1}{(\sin t)^{1-j}}\,dt\leq \frac{2^{1-j}}{j}  \biggl( \Bigl( \frac{\pi}{2}-\arctan\frac{\sqrt{\sigma}}{\sigma}\Bigr)^{j}
  -\Bigl( \frac{\pi}{2}-\arctan\frac{T}{\sigma}\Bigr)^{j}\biggr).
\]
As a result, $I_j\to0$ as $\sigma\to0+$, that, together with \eqref{eq:sdasd111} and  \eqref{Bsigmaest2}, proves that   $B_j(\sigma) \to 0$, $\sigma\to  0+$.

Combining this with $A_j(\sigma)\to 0$, one gets \eqref{eq:eta} from \eqref{ABexpans}; and we can apply Proposition~\ref{lem:EII}. Namely, by \eqref{eq:alpha},
there exist $C>0$ and $t_0\geq1$, such that
\[
  \frac{D}{\Gamma(j)} e^{at}t^{j-1}\leq \varphi(t)e^{(\mu+a)t} \leq \left\{ \frac{D}{\Gamma(j)}+C\rho(t) \right\}e^{at}t^{j-1},\quad t\geq t_0.
\]
By \eqref{eq:eta} and \eqref{eq:rho}, $\rho(t)\to  0$ as $t\to \infty$. Therefore,
\[
  \varphi(t)e^{(\mu+a)t}\sim \frac{D}{\Gamma(j)} e^{at}t^{j-1},\quad t\to  \infty,
\]
that is equivalent to \eqref{eq:psi_asympt} and finishes the proof.
\end{proof}

The following simple Proposition show that if $F$ in \eqref{eq:sing_repres} is holomorphic on the line $\{\Re z =\mu\}$, then \eqref{loglim} holds.
\begin{proposition}\label{prop:Carr_Chmaj}
Let $\varphi:\R_+\to  \R_+$ be a non-increasing function such that, for some $\mu>0$, $\nu>0$,
\eqref{assum:1}--\eqref{eq:psi_onesided_lap} hold.
  Suppose also that there exist $j>0$ and $F,H\in \mathcal{A}(0<\Re z\leq \mu)$, such that
  \eqref{eq:sing_repres} holds.
  Then $\varphi$ has the asymptotic \eqref{eq:psi_asympt}.
\end{proposition}
\begin{proof}
Take any $\beta\in(0,\mu)$ and $T>0$. Let $K_{\beta,\mu,T}$ be defined by \eqref{eq:Kset}.
  Since $F\in \mathcal{A}(0<\Re z\leq \mu)$, then $F'\in C(K_{\beta,\mu,T})$, and hence $F'$ is bounded on $K_{\beta,\mu,T}$. Then one can apply a mean-value-type theorem for complex-valued functions, see e.g. \cite[Theorem~2.2]{EJ1992}, to get that, for some $K>0$,
\[
|F(\mu-2\sigma-i\tau)-F(\mu-\sigma-i\tau)|\leq K |\sigma|, \qquad 2\sigma<\mu- \beta,
\]
that yields \eqref{loglim} for all $j>0$, cf.~\eqref{eq:def_gj}. Hence we can apply Theorem~\ref{thm:tauber}.
\end{proof}

\begin{remark}
Note that, for $F\in \mathcal{A}(0<\Re z\leq \mu)$  in \eqref{eq:sing_repres}, the holomorphic function $H$ is redundant there, as we always can replace $F(z)$ by a holomorphic function $F(z)+H(z)(\mu-z)^j$. Therefore, Proposition~\ref{prop:Carr_Chmaj} corresponds to \cite[Proposition~2.3]{CC2004}.
\end{remark}

\begin{proposition}\label{prop:Diekmann_Kaper}
Let $\varphi:\R_+\to  \R_+$ be a non-increasing function such that, for some $\mu>0$, $\nu>0$,
\eqref{assum:1}--\eqref{eq:psi_onesided_lap} hold. Suppose also that there exist $j\geq1$, $D>0$,  and  $h:\R\to\R$ such that
    \begin{equation}\label{eq:DK}
            H(z):=\int_0^\infty e^{z t} \varphi(t) dt - \frac{D}{(\mu-z)^j} \to h\bigl(\Im z \bigr), \quad \Re z\to\mu -,   
    \end{equation}
uniformly (in $\Im z$) on compact subsets of $\R$. 
  Then the following asymptotic holds, 
  \begin{equation}\label{eq:Diekmann_Karper_type}
    \varphi(t) \sim \frac{D}{\Gamma(j)} t^{j-1} e^{-\mu t},\quad t\to\infty.
  \end{equation}
\end{proposition}
\begin{proof}
Let $a>\max\{0,\nu-\mu\}$ and $\alpha(t)$ be given by \eqref{eq:alpha-new}. 
Let $G$ be given by~\eqref{eq:G}. Similarly to the proof of Theorem~\ref{thm:tauber}, we will get from \eqref{eq:psi_i} and \eqref{eq:DK}, that
  \begin{equation*}
    G(z) =H(\mu-z)-\dfrac{\varphi(0)}{a+z},\quad 0<\Re z<\mu.
  \end{equation*}
  Next, 
  \eqref{eq:DK} 
  implies \eqref{eq:eta}. 
  Hence, by Lemma~\ref{lem:EII}, \eqref{eq:Diekmann_Karper_type} holds that fulfilled the proof.
\end{proof}

Note that the result in \cite[Lemma~6.1]{DK1978} corresponds to $j=1$ in Proposition~\ref{prop:Diekmann_Kaper}.

\begin{remark}
It is worth noting that, for the case $j>1$, we have, by \eqref{eq:etaform}, that if $G$ is bounded, then \eqref{eq:eta} holds. Therefore, in this case, it is enough to assume that $H$ in \eqref{eq:DK} is bounded to conclude \eqref{eq:Diekmann_Karper_type}.
\end{remark}


\begin{thebibliography}{10}

\bibitem{CC2004}
J.~Carr and A.~Chmaj.
\newblock Uniqueness of travelling waves for nonlocal monostable equations.
\newblock {\em Proc. Amer. Math. Soc.}, 132\penalty0 (8):\penalty0 2433--2439, 2004.

\bibitem{CDM2008}
J.~Coville, J.~D{\'a}vila, and S.~Mart{\'{\i}}nez.
\newblock Nonlocal anisotropic dispersal with monostable nonlinearity.
\newblock {\em J. Differential Equations}, 244\penalty0 (12):\penalty0
  3080--3118, 2008.

\bibitem{DK1978}
O.~Diekmann and H.~G. Kaper.
\newblock On the bounded solutions of a nonlinear convolution equation.
\newblock {\em Nonlinear Anal.}, 2\penalty0 (6):\penalty0 721--737, 1978.

\bibitem{EE1985}
W.~Ellison and F.~Ellison.
\newblock {\em Prime numbers}.
\newblock A Wiley-Interscience Publication. John Wiley \& Sons Inc., New York;
  Hermann, Paris, 1985.
\newblock xii+417~pp.

\bibitem{EJ1992}
J.-C. Evard and F.~Jafari.
\newblock A complex {R}olle's theorem.
\newblock {\em Amer. Math. Monthly}, 99\penalty0 (9):\penalty0 858--861, 1992.

\bibitem{FKT100-2}
D.~Finkelshtein, Y.~Kondratiev, and P.~Tkachov.
\newblock Doubly nonlocal {F}isher--{KPP} equation: Speeds and uniqueness of traveling waves.
\newblock 2018.

\bibitem{Kor2002}
J.~Korevaar.
\newblock A century of complex {T}auberian theory.
\newblock {\em Bull. Amer. Math. Soc. (N.S.)}, 39\penalty0 (4):\penalty0
  475--531, 2002.

\bibitem{LZ2016}
T.~S. Lim and A.~Zlato{\v{s}}.
\newblock Transition fronts for inhomogeneous {F}isher-{KPP} reactions and
  non-local diffusion.
\newblock {\em Trans. Amer. Math. Soc.}, 368\penalty0 (12):\penalty0
  8615--8631, 2016.

\bibitem{Ten1995}
G.~Tenenbaum.
\newblock {\em Introduction to analytic and probabilistic number theory},
  volume~46 of {\em Cambridge Studies in Advanced Mathematics}.
\newblock Cambridge University Press, Cambridge, 1995.
\newblock xvi+448~pp.
\newblock Translated from the second French edition (1995) by C. B. Thomas.

\bibitem{TPT2016}
E.~Trofimchuk, M.~Pinto, and S.~Trofimchuk.
\newblock Monotone waves for non-monotone and non-local monostable
  reaction--diffusion equations.
\newblock {\em J. Differential Equations}, 261\penalty0 (2):\penalty0
  1203--1236, 2016.

\bibitem{WLX2018}
P.~Weng, L.~Liu, and Z.~Xu.
\newblock Monotonicity, asymptotic behaviors and uniqueness of traveling waves
  to a nonlocal dispersal equation modeling an age-structured population.
\newblock {\em Nonlinear Anal. Real World Appl.}, 39:\penalty0 58--76, 2018.

\bibitem{Wid1941}
D.~V. Widder.
\newblock {\em The {L}aplace {T}ransform}.
\newblock Princeton Mathematical Series, v.~6. Princeton University Press,
  Princeton, N. J., 1941.
\newblock x+406~pp.

\bibitem{ZLW2012}
G.-B. Zhang, W.-T. Li, and Z.-C. Wang.
\newblock Spreading speeds and traveling waves for nonlocal dispersal equations
  with degenerate monostable nonlinearity.
\newblock {\em Journal of Differential Equations}, 252\penalty0 (9):\penalty0
  5096--5124, 2012.

\end{thebibliography}
\end{document}